\documentclass[a4paper,10pt,reqno, english]{amsart}

\usepackage{amsmath,amssymb,amscd,amsthm,amsfonts}
\usepackage{graphicx,subfigure}
\usepackage{hyperref}
\usepackage{dsfont}
\usepackage{xcolor}
\usepackage[utf8]{inputenc}
\usepackage[nobysame, alphabetic]{amsrefs}

\newtheorem{theorem}{Theorem}[subsection]

\newtheorem{claim}[theorem]{Claim}
\newtheorem{corollary}[theorem]{Corollary}

\newtheorem{problem}[theorem]{Problem}
\newtheorem{definition}{Definition}

\newcommand{\rr}{\mathds{R}}

\DeclareMathOperator{\conv}{conv}

\title{Tverberg's theorem, disks, and Hamiltonian cycles}

\author[Sober\'on]{Pablo Sober\'on}\address{Baruch College, City University of New York, One Bernard Baruch Way, New York, NY 10010} 
\email{pablo.soberon-bravo@baruch.cuny.edu}

\author[Tang]{Yaqian Tang}
\address{Wesleyan University, Middletown, CT 06459}
\email{ytang@wesleyan.edu}

\thanks{This research project was done as part of the 2020 Baruch Discrete Mathematics REU, supported by NSF awards DMS-1802059, DMS-1851420, and DMS-1953141.  Sober\'on's research is also supported by PSC-CUNY grant 63529-00 51.  Tang's research was supported by Wesleyan University's Sumer Science Research Endowed Fund.}

\begin{document}

\begin{abstract}
For a finite set $S$ of points in the plane and a graph with vertices on $S$ consider the disks with diameters induced by the edges.  We show that for any odd set $S$ there exists a Hamiltonian cycle for which these disks share a point, and for an even set $S$ there exists a Hamiltonian path with the same property.  We discuss high-dimensional versions of these theorems and their relation to other results in discrete geometry.
\end{abstract}

\maketitle

\section{Introduction}

In 1966, Helge Tverberg proved that \textit{for any set of $(r-1)(d+1)+1$ points in $\rr^d$ there exists a partition of them into $r$ parts whose convex hulls intersect} \cite{Tverberg:1966tb}.  We call these partitions \emph{Tverberg partitions}.  Among the variations and generalizations of Tverberg's theorem, two kinds stand out.  In the first, we impose additional conditions on the partitions or try to deduce structural properties of the family of all Tverberg partitions of a set.  The colorful versions of Tverberg's theorem, or Sierksma's conjecture on the number of Tverberg partition fall into this category.  In the second, we relax or modify the geometric conditions while not breaking the existence of Tverberg's partitions.  The topological versions of Tverberg's theorem are an example of such extensions.  We recommend \cites{Barany:2016vx, Barany:2018fya, DeLoera:2019jb} and the references therein for the developments around Tverberg's theorem.

This manuscript focuses on a variation of the second kind.  For a segment $e$ in $\rr^d$ with endpoints $x,y$, we denote by $D(e)$ or $D(x,y)$ the closed ball for which $e$ is a diameter.  Given a finite set of points in $\rr^d$, instead of looking at the the convex hulls of its subsets, we are interested in the balls spanned by pairs of its points.  For any graph $G$ with vertices in $\rr^d$, we consider a rectilinear drawing of $G$, where each segment is represented by a straight segment.

\begin{definition}\label{def:tverberg-graph}
	Let $S$ be a finite set of points in $\rr^d$.  Let $G$ be a graph whose vertex set is $S$ and whose edge set is $E$.  We say that $G$ is a \emph{Tvereberg graph} for $S$ if
	\[
	\bigcap_{e\in E}D(e) \neq \emptyset.
	\]
\end{definition}

\begin{figure}
\centerline{\includegraphics[width = \textwidth]{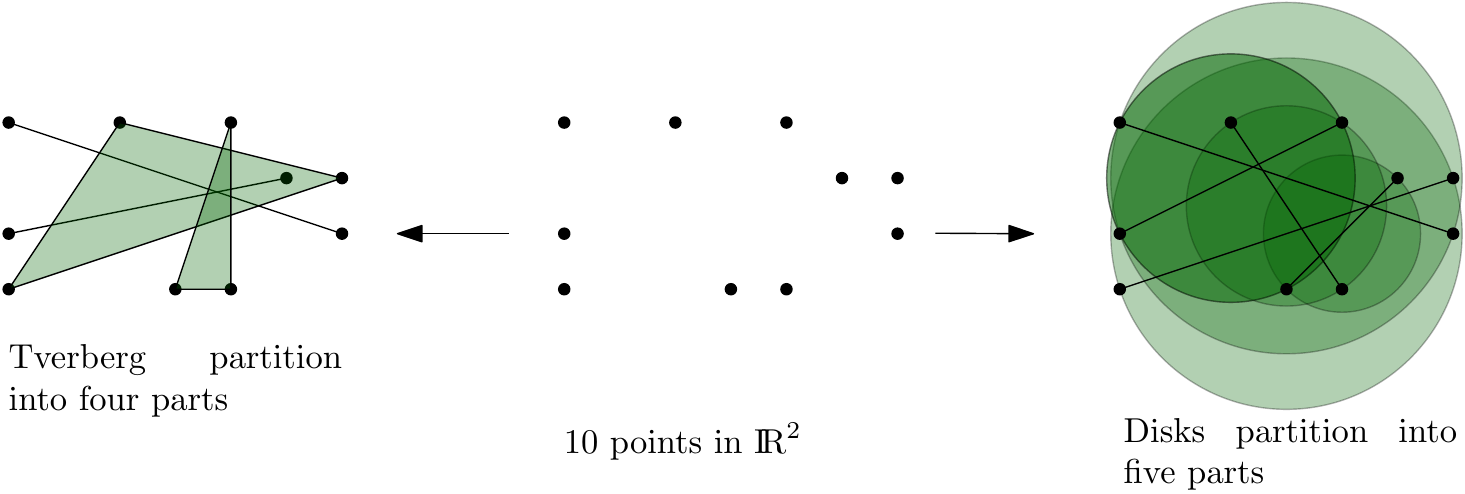}}	
\caption{Two different partitions for the same set of points.}
\end{figure}

A natural way to extend Tverberg's theorem to disks is to ask if for any $2r$ points in $\rr^d$ there exists a perfect matching that is a Tverberg graph.  In the plane, Huemer, P\'erez-Lantero, Seara, and Silveira proved said result \cite{Huemer:2019ji}.  They even showed that for any $r$ blue point and any $r$ red point there is a perfect red-blue matching that is a Tverberg graph.  This extends the colorful Tverberg theorem to disks.  Bereg, Chac\'on-Rivera, Flores-Pe\~naloza, Huemer, and P\'erez-Lantero found a second proof of the monochromatic version \cite{bereg2019maximum}.  Definition \ref{def:tverberg-graph} leads us to the following problem.

\begin{problem}\label{prob:main-problem}
	Given a finite set of points $S$ in $\rr^d$, determine the family of Tverberg graphs for $S$.
\end{problem}

There is no need for the graphs from Problem \ref{prob:main-problem} to be matchings.  Once $S$ is fixed, Tverberg graphs are closed under containment, so finding the containment-maximal Tverberg graphs is of interest.  In this manuscript, we show that the family of Tverberg graphs for a set of points always contains other interesting graphs.  We are particularly interested in Hamiltonian cycles and Hamiltonian paths.

\begin{theorem}\label{thm-odd}
	Let $S$ be a finite set of points in the plane.  If $S$ has odd cardinality there exists a Hamiltonian cycle that is a Tverberg graph for $S$.  If $S$ has even cardinality, there exists a Hamiltonian path that is a Tverberg graph for $S$.
\end{theorem}

The planar result by Huemer et al. and by Bereg et al. rely on a particular choice of a perfect matching and then a reduction of the problem with $2r$ points to a problem with six points via Helly's theorem.  Such a reduction does not work for Hamiltonian cycles or paths, since the property of being a Hamiltonian cycle cannot be verified locally.  Our methods rely instead on choosing the point of intersection of the disks and then constructing the cycle.

We note that the essence of Theorem \ref{thm-odd} is the existence of Hamiltonian cycles when $|S|$ is odd.  This result implies the existence of the Hamiltonian path for $|S|$ even, as we show below.

\begin{corollary}
Let $S$ be a set of $2r$ points in $\rr^2$.  There exists a Hamiltonian path that is a Tverberg graph for $S$.
\end{corollary}

\begin{proof}
       We append an additional point $x$ to $S$.  By the first part of Theorem \ref{thm-odd}, there exists a Hamiltonian cycle that is a Tverberg graph for $S\cup \{x\}$.  By removing $x$ and its two edges, we are left with a Hamiltonian path on $S$, as desired.
\end{proof}

A Hamiltonian path contains a perfect matching, so this corollary implies the most simple version of ``Tverberg for disks''.  The trick of appending an additional point is quite useful in other Tverberg-type problems \cite{Blagojevic:2014js, Blagojevic:2011vh, Blagojevic:2015wya, Frick:2020uj}.  If the appended point $x$ is a vertex of $S$, then we get a slightly stronger result.  There is a Tverberg graph on $S$ which is the union of two cycles sharing exactly one vertex.  The union of the two cycles contains every point in $S$, and the shared vertex may be fixed in advance.

The reason our results work in dimension two is because we use an idea similar to Birch's proof of Tverberg's theorem in the plane \cite{Birch:1959ii}.  Other Tverberg-type theorems are much better understood in the plane \cite{Barany:1992tx}

In high dimensions we obtain the following results, which provides a clear connection between Tverberg graphs (for disks) and Tverberg's original theorem.

\begin{theorem}\label{thm-tverbergapplied}
    Let $S$ be a finite set of points in $\rr^d$ such that $|S| \ge (r-1)(d+1)+1$.  There exists a partition of $S$ into $r$ non-empty sets $A_1,\ldots, A_r$ and a Tverberg graph with vertices of $S$ such that every point of $S$ is adjacent to at least one point of each $A_i$. 
\end{theorem}

In particular, this result shows that every finite set in $\rr^d$ has Tverberg graphs which are dense.

\begin{corollary}
	Let $S$ be a finite set of points in $\rr^d$.  There is a Tverberg graph for $S$ whose minimum degree is greater than or equal to $|S|/(d+1)$.
\end{corollary}

We set notation and preliminaries in Section \ref{sec:notation} and prove Theorem \ref{thm-odd} in Section \ref{sec:main-proofs}.  Finally, we prove Theorem \ref{thm-tverbergapplied} in Section \ref{sec:high-dimensions}.  We present open problems throughout the manuscript.

\section{Notation and general position assumptions}\label{sec:notation}

In our proofs, we assume that our set of points $S$ is in general position.  In this paper, we say that a finte set $S$ in $\rr^2$ is in general position if:
\begin{itemize}
    \item no three points of $S$ are collinear,
    \item for any $x,y,z \in S$, we know $z$ is not in the boundary of $D(x,y)$ if $z \not\in\{x,y\}$, and
    \item for any three pairs $\{x_1, y_1\}, \{x_2, y_2\}, \{x_3, y_3\}$ of points of $S$ the boundaries of $D(x_1, y_1)$, $D(x_2, y_2)$, and $D(x_3,y_3)$ do not intersect unless $\bigcap_{i=1}^3 \{x_i,y_i\} \neq \emptyset$.
    \item If $\{x_1,y_1\}, \{x_2,y_2\}$ are two pairs of points of $S$, then the boundaries of $D(x_1, y_1), D(x_2,y_2)$ are not tangent.
\end{itemize}

It is clear that any $n$-tuple of points $S$ can be approximated by a sequence $\{S_i\}$ of $n$-tuples of points in general position.  If we can find a particular Tverberg graph $G_i$ for each $S_i$, a subsequence of those graphs will determine the same adjacent pairs.  Therefore, the corresponding limiting graph $G$ will be a Tverberg graph for $S$.  If each $G_i$ satisfies an additional property (such as being regular, a Hamiltonian cycle or satisfying a bound on its minimum degree), so will $G$.

Given a graph $G$, we denote by $E(G)$ its edges and by $V(G)$ its vertices.

\section{An odd number of points in the plane}\label{sec:main-proofs}

The particular case when $S$ is in convex position showcases some of the obstacles in solving this problem.  For two points $x, y \in \rr^2$ and $\alpha \in (0,\pi)$ we define the $\alpha$-lens $\alpha(x,y)$ as 
\[
\alpha(x,y) =\{z \in \rr^d : \angle xzy \ge \alpha \}.
\]

If $\alpha = \pi/2$, then $\alpha(x,y) = D(x,y)$.  Alpha-lenses have interesting intersection porperties \cite{barany1987extension, Barany:1987ef, Magazinov:2016ia}.  If $e$ denotes the segment $xy$, we use the notation $\alpha(e) = \alpha(x,y)$.

\begin{theorem}
	Let $S$ be a finite set of points in convex position in the plane.  If $|S|$ is odd, then there exists a Hamiltonian cycle $G$ with vertex set $S$ such that 
	\[
	\bigcap_{e \in E(G)} D(e) \neq \emptyset.
	\]
\end{theorem}

\begin{proof}
	Let $|S| = 2n+1$. Order the points $x_1, x_2, \ldots, x_{2n+1}$ clockwise along the the boundary of the convex hull of $S$.  The Hamiltonian cycle is formed by making $x_i$ adjacent to $x_{i+n}$ and $x_{i+n+1}$, where the indices are considered modulo $2n+1$. By construction, any two of the segments intersect, so any three segments contain the sides of a triangle.  If $e_1, e_2, e_3$ are three edges of $G$, let $s_1, s_2, s_3$ be the sides of the triangle so that $s_i \subset e_i$ for each $i$.
	
	Notice that $D(s_1) \cap D(s_2) \cap D(s_3) \neq \emptyset$ since the foot of the height from the largest angle is contained in the opposite side, and therefore it is in each disk.  Since $D(s_i) \subset D(e_i)$, we know that $D(e_1) \cap D(e_2) \cap D(e_3) \neq \emptyset$.  Finally, by Helly's theorem we obtain that $G$ is a Tverberg graph.
\end{proof}

The convex position scenario shows why the case when $|S|$ is odd behaves differently from the case when $|S|$ is even.  The proof above works for $\alpha$-lenses with $\alpha = 2\pi / 3$.  This is because for a triangle $xyz$ in the plane, either one of the angles is at least $2\pi/3$, or the Fermat-Toricelli point lies inside the triangle.  The Fermat-Toricelli point $p$ is the point that minimizes the sum of the distance to the vertices, and has the property that the angles $\angle xpy, \angle ypz, \angle zpx$ are all equal to $2\pi/3$.  This implies that the alpha-lenses $\alpha(x,y), \alpha(y,z), \alpha(x,z)$ intersect for $\alpha = 2\pi/3$.

For $|S|$ even, one cannot hope for such an extension to $\alpha$-lenses.  If $S$ is the set of vertices of a square, for any Hamiltonian cycle $G$, we have $\bigcap_{e\in E(G)}\alpha (e) = \emptyset$ for $\alpha> \pi/2$.  The extension to $\alpha$-lenses cannot does not hold for odd sets of points in general.  If $S$ is the set of vertices of a square and its center, then for any $\alpha > \pi/2$ the $\alpha$-lenses induced by any Hamiltonian cycles do not all intersect.

The general case of Theorem \ref{thm-odd}, when $S$ is not in convex position, is not so simple.  We do not have a canonical way to order the points and make a Hamiltonian cycle.  We circumvent this problem by creating Hamiltonian cycles ``around a point $p$''.  We choose a point $p$ as a candidate for the intersection of the disks.  Notice that for any two points $x,y \in S$, we have
\[
p \in D(x,y) \iff \angle xpy \ge \pi/2.
\]

We assume that our set of points $S$ is in general position.  For a point $p \in \rr^2$, we construct a Hamiltonian cycle for $p$ as follows:

\begin{itemize}
	\item If $p \not\in S$.  Consider $C(p)$ be the circle of radius $1$ around $p$.  We project radially every point of $S$ onto $C(p)$.  We label them $x_1, \ldots, x_{2n+1}$ clockwise around $C(p)$.  We form our Hamiltonian cycle by connecting $x_i$ with $x_{i+n}$ and $x_{i+n+1}$.  The cycle does not depend on the choice of $x_1$.  We call this a Hamiltonian cycle of type I.  See Figure \ref{fig:type1}.
	\item If $p \in S$.  Consider $C(p)$ the circle of radius $1$ around $p$.  We choose a point $p^* \in C(p)$ to represent $p$.  We project radially every other point of $S$ onto $C(p)$.  We label the points $x_1, \ldots, x_{2n+1}$ clockwise around $C(p)$ so that $p^* = x_{2n+1}$.  We form the Hamiltonian cycle by connecting $x_i$ with $x_{i+n}$ and $x_{i+n+1}$.  The cycle only depends on the choice of $p^*$. We call this a Hamiltonian cycle of type II.  See Figure \ref{fig:type2}.
\end{itemize}

\begin{figure}
	\centerline{\includegraphics{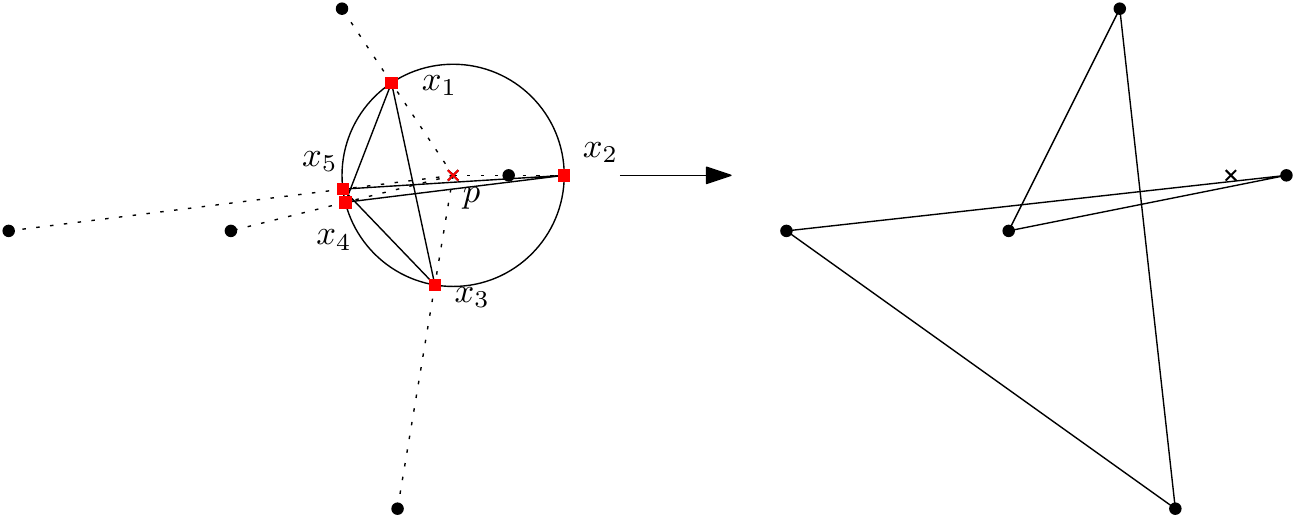}}
	\caption{A type I Hamiltonian cycle by five points from a center $p \not\in S$}
	\label{fig:type1}
\end{figure}

\begin{figure}
	\centerline{\includegraphics{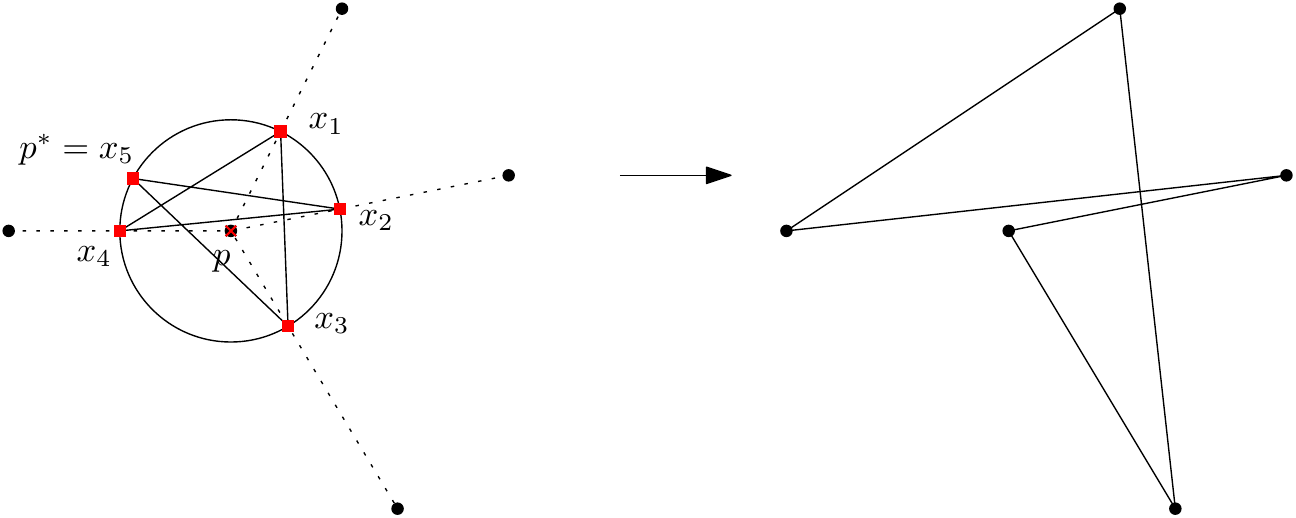}}
	\caption{A type II Hamiltonian cycle by five points from a center $p \in S$.  The choice of representative $p^* \in C(p)$ is important.  }
	\label{fig:type2}
\end{figure}

Notice that $p \in \bigcap_{e \in E(G)}D(e)$ where $E(G)$ is the set of edges of the Hamiltonian cycle of Type I if and only if $\angle x_i p x_{i+n} \ge \pi /2$ for each $i\in [2n+1]$.  For Hamiltonian cycle of type II, we only need to check $\angle x_i p x_{i+n} \ge \pi /2$ for $i \in [2n] \setminus \{n+1\}$, since the disks from edges ending at $p$ contain $p$.  Even though different points may induce the same cycles, we are checking if $p$ is in the disks induced by all edges of the cycle.  In Figure \ref{fig:type1} the answer would be negative since $p \not\in D(x_3,x_5)$, and in Figure \ref{fig:type2} the answer would be positive.

Given a point $p \in \rr^2 \setminus S$ we define $\ell(p)$ as the number of indices $i\in [2n+1]$ such that the angle $\angle x_i p x_{i+n}$ is strictly smaller than $\pi /2$ for the type I Hamiltonian cycle around $p$.  For a point $p \in S$, we define $\ell(p)$ as the smallest number of indices $i \in [2n] \setminus\{n+1\}$ such that the angle $\angle x_i p x_{i+n}$ is strictly smaller than $\pi/2$ for all Hamiltonian cycles of type II around $p$.  The parameter $\ell(p)$ is the number of disks that do not contain $p$.  If $\ell(p) = 0$ for some $p$, then we have a Hamiltonian cycle as we wanted.  For any $k \in \{0,1, \ldots, 2n+1\}$, the set of points $p$ such that $\ell (p) \le k$ is a closed set, which can be checked by sequences.  For $k < 2n+1$, it is contained in the union of all disks spanned by pairs of  $S$, and is therefore a compact set.

Now, for a value $k \in \{0,1,\ldots, 2n+1\}$ we consider $C_k \subset \rr^2$ as the region of the plane defined by
\[
C_k = \{p \in \rr^2 : \ell(p) \le k \}.
\]

Let $k$ be the smallest value such that $C_k \neq \emptyset$.  We assume $k \ge 1$, and look for a contradiction.  In this case, $C_k$ is compact.

For a point $p \in C_k \setminus S$ we define a function $f(p)$ as the sum of the $k$ angles $\angle x_i p x_{i+n}$ which are strictly smaller than $\pi / 2$.  If $p \in C_k \cap S$, we take the largest such sum among all type II Hamiltonian cycles around $p$, and don't consider the values $i=n+1, i=2n+1$ in the sum.

\begin{claim}\label{claim:odd1}
	The function $f$ attains a maximum value in $C_k$.
\end{claim}

\begin{proof}
	Notice that $f$ is continuous in $C_k \setminus S$.  Let $M = \max \{f(p) : p \in S \cap C_k\}$.  If $M$ is not the maximum, then there exists a point $p_0 \in C_k$ such that $f(p_0) > M$.  Let $R = \{p \in C_k : f(p) \ge f(p_0)\}$.  Let us show that no point of $S$ is in the closure of $R$.  If that's not the case, there is a sequence $p_1, p_2, \ldots$ in $R$ that converges to a point $p$ in $S$.  The sequence of points of the form
	\[
	q_i = p + \frac{1}{||p-p_i||}(p-p_i)
	\]
	lies in $C(p)$.  Since $C(p)$ is compact, by taking subsequences we can assume that $q_i$ converges to some point $q^*$.  The type I Hamiltonian cycle induced by $p_i$ is eventually equal to the type II Hamiltonian cycle induced by $p$ using the reprresentative $q^*$ for $p$.  Therefore, the limit of $f(p_i)$ is equal to the function evaluated at $p$ with representative $q^*$.  The value of $f(p)$ is greater than or equal to this value.  This contradicts the fact that $f(p_i) \ge f(p_0)$ for each $i \ge 1$.
	
	Since $S$ is not in the closure of $R$, $f$ is continuous in $C_k \setminus S$, and $C_k$ is compact, then $R$ is compact.  The function $f$ must have a maximum value in $R$, which is a maximum value over all of $C_k$.
\end{proof}

The following claim will give us the contradiction we seek, showing that $k=0$.

\begin{claim}\label{claim:ood2}
	If $k \ge 1$, the function $f$ cannot attain a maximum value in $C_k$.
\end{claim}

\begin{proof}
	First, consider the case $p \not\in S$.  Consider the set of arcs $x_ix_{i+n}$ on $C(p)$ such that $\angle x_i p x_{i+n} \le \pi/2$.  Of the two circular arcs determined by $x_i, x_{i+n}$ we are considering the smaller one.  There are $k$ arcs which make an angle strictly smaller than $\pi/2$ and at most two that make an angle equal to $\pi/2$ due to the general position assumptions.
	
	If we include its extremes, each of these arcs contains at least $n+1$ points $x_i$, possibly more if the extremes $x_i, x_{i+n}$ coincide other projections.  If two of these arcs were disjoint, there would be at least $2n+2$ points in $S$, a contradiction.  Therefore, any two of these arcs intersect.  For any three of these arcs, their union cannot cover $C(p)$, or at least one of them would make an angle of at least $2\pi/3$.  Therefore, any three of these arcs intersect.
	
	Now, for each of these arcs $x_ix_{i+n}$, we define the convex set $K_i$ formed by the convex hull of $p$ and the circular arc $x_ix_{i+n}$.  Since $p$ is an extreme point of $K_i$, the set $K_i\setminus \{p\}$ is convex.  By Helly's theorem, since every three sets of the sets of the form $K_i \setminus\{p\}$ intersect, there is a point $y^*$ common to all of them.  The radial projection $x^*$ of $y^*$ onto $C(p)$ is in all the arcs $x_ix_{i+n}$ we were considering.
	
	If we move $p$ in the direction $x^*-p$ a small enough distance, we remain in $C_k$ and the value of $f(p)$ increases.  Therefore, $p$ was not a maximum of $f$.
	
	Now, consider the case $p \in S$.  We assume without loss of generality that $p$ is the origin, so antipodal points in $C(p)$ are negative of each other.  Let $p^*$ be a representative of $p$ in $C(p)$ that realizes $f(p)$.  
		
	No arcs $x_i x_j$ for $i,j \neq 2n+1$ have an angle equal to $\pi/2$ due to the general position argument.  As before, the arcs $x_ix_{i+n}$ which form angles smaller than $\pi/2$ all have a nonempty intersection $Q$.  We first show that we can assume $-p^* \in Q$.
	
	If $-p^*$ was not a point in the intersection of all the small arcs, let $q \in Q$ be an arbitrary point in that intersection.  We replace $p^*$ by $-q$ and reconstruct the Hamiltonian cycle.  Denote by $\tilde{x}_1, \ldots, \tilde{x}_{2n}$ the new numbering of the points in $S \setminus p$.  If there is no new arc $\tilde{x}_i \tilde{x}_{i+n}$ that contains $-q$ and whose angle is smaller than $\pi/2$, all the arcs we would consider now are either arcs of the form $x_ix_{i+n}$ that did not contain $p^*$ or a widening of one of such arcs.
	
	Therefore, we may assume that there is an arc $\tilde{x}_i\tilde{x}_{i+n}$ that contains $-q$ and has angle smaller than $\pi/2$.  Let $A$ be the set of all such arcs.  Let $B$ be the set of all arcs $x_jx_{j+n}$ in the first Hamiltonian cycle that had angle smaller than $\pi/2$.  We know that $(\cup A) \cap (\cup B) = \emptyset$, because any arc in $B$ contains $q$, any arc in $A$ contains $-q$, and both had angles smaller than $\pi/2$.  The arc $\cup A$ contains at least $n+|A|-1$ projections of points of $S$ (not counting $q$).  The arc $B$ contains at least $n+|B|-1$ projections of points of $S$.  In total this gives us at least $2n +|A|+|B|-2$. Since there is a total of $2n$ projections of points of $S$ onto $C(p)$, this means that $|A|=|B| = 1$.  
	
	Therefore, for $k \ge 2$ we can replace $p^*$ by $-q$ without decreasing $f(p)$.  If $k=1$, the (single) arc $A$ as defined above contains $n$ of the projections $x_i$ and the arc $B$ contains another $n$ (and therefore must contain $p^*$, or it would need to contain at least $n+1$ of the projections).  There are no points of the form $x_j$ outside $A \cup B$.  The set $C(p) \setminus (A \cup B)$ is the union of two arc whose angles sum to more than $\pi$, so one of them must have angle greater than $\pi/2$.  Let $p^{**}$ be a point in this arc.  We replace $p^*$ by $p^{**}$.  When we make the Hamiltonian cycle with this new representative, the arc $B$ either widens or gets removed.  In the first case, $f(p)$ increases, and in the second we would have $p \in C_0$.  Notice that no new short arc containing $p^{**}$ was created because the two projection points $x_i, x_j$ next to $p^{**}$ form an angle greater than $\pi/2$.
	
	This means we can always assume that $-p^*$ is in the intersection of all the arcs of the form $x_{i}x_{i+n}$ with angle smaller than $\pi/2$.  Finally, let us show that $\angle x_npp^* > \pi/2$ and $\angle p^*p x_{n+1} > \pi/2$.   We show the bound on the first angle, and the second is analogous.  If $\angle x_n p p^* \le \pi/2$, notice that the arc $A = x_n p^*$ contains at least $n$ points of the form $x_i$.  Take an arc $B$ of the form $x_i x_{i+n}$ with angle smaller than $\pi/2$.  This arc contains $-p^*$ and therefore cannot intersect $A$.  Moreover, $B$ must contain at least $n+1$ points of the form $x_i$.  Then, $A \cup B$ contain at least $2n+1$ such points, but there are only $2n$ of them.
	
	With these conditions, we can now move $p$ a small distance in the direction $p-p^*$, reaching a new point $p'$.  Since $p\in S$, when we represent it in $C(p')$ it will be the point in direction of $p^*$.  This means that the canonical Hamiltonian cycle of type I will be the same that we had with $p$ when the representative was $p^*$.  Since $-p^*$ was in the intersection of all the short arcs, they all widen when we reach $p'$.  Therefore, $f(p)$ was not maximal.
\end{proof}

\section{Results in high dimensions and remarks}\label{sec:high-dimensions}

We show how to use Tverberg's theorem to prove Theorem \ref{thm-tverbergapplied}.

\begin{proof}
Let $S$ be a set of at least $(r-1)(d+1)+1$ points in $\rr^d$.  By Tverberg's theorem, there exists a partition of the points into $r$ sets $A_1, \ldots, A_r$ whose convex hulls intersect.  Let $p$ a point the point of intersection of $\conv A_j$ of $j=1,\ldots, d$.

Given a point $q \in S$ and a value $1\le j \le r$, consider the  half-space 
\[
H^+ = \{x \in \rr^d: \langle x, q-p\rangle \ge \langle p, q-p\rangle \}.
\]
By definition, $p \in H^+$.  Therefore, since $p \in \conv A_j$, there is an element $q_j \in A_j \cap H^+$.  Notice that $p \in D(q,q_j)$.  We construct a graph $G$ with vertices in $S$ and make $\{q,q_j\}$ one of the edges.  We repeat this process for each $q \in S$ and each $j =1,\ldots, r$.  The graph we constructed satisfies the properties of the problem.
\end{proof}

A general unsolved problem is the following.

\begin{problem}
    Let $S$ be a set of points in $\rr^d$.  Determine if there exists a Hamiltonian cycle with vertices on $S$ that is a Tverberg graph for $S$.
\end{problem}

We suspect that the answer should be positive in the plane, regardless of the parity of $S$.  In high dimension we don't even know if, for even $|S|$, there exists a perfect matching that is a Tverberg graph for $S$.

As an additional bit of evidence that the answer should be positive in the plane, consider the case of four points in $\rr^2$.  If the convex hull of the points is a triangle with vertices $x,y,z$, notice that every point inside $\triangle xyz$ is covered twice by the sets $D(x,y), D(y,z), D(x,z)$.  Assume without loss of generality that the point $w$ is covered by $D(x,y), D(x,z)$.  Then, the Hamiltoninan cycle $(w,y,x,z,w)$ is a Tverberg graph (the point $w$ is in the intersection of all disks).  

If the convex hull of the four points is a convex quadrilateral with intersecting diagonals $xy, wz$, let $p$ be the point of intersection of the diagonals.  One of the two angles $\angle xpw$ and $\angle wpy$ must be greater than or equal to $\pi/2$.  If, without loss of generality, $\angle xpw \ge \pi/2$, then the Hamiltonian cyclee $(x,w,z,y,x)$ is a Tverberg graph (the point $p$ is in the intersection of all disks).



\end{document}